\documentclass[11pt]{article}
\usepackage[utf8]{inputenc}
\usepackage{amssymb}
\usepackage{amsmath}
\usepackage{bbm}
\usepackage{mathrsfs}
\usepackage{amsthm}
\usepackage{theoremref}
\usepackage{color}
\usepackage{enumitem}
\usepackage[top=1in, bottom=1in, left=1in, right=1in]{geometry}
\usepackage{setspace}
\usepackage[round]{natbib}
\usepackage{mathtools}

\usepackage{multirow}
\usepackage{algorithm}
\usepackage[noend]{algpseudocode}
\usepackage{authblk}
\usepackage{appendix}

\providecommand{\keywords}[1]
{
  \small	
  \textbf{\textit{Keywords---}} #1
}

\makeatletter
\def\BState{\State\hskip-\ALG@thistlm}
\makeatother

\newtheorem{lemma}{Lemma}
\newtheorem{corollary}{Corollary}
\newtheorem{proposition}{Proposition}

\newtheorem{remark}{Remark}
\newtheorem{definition}{Definition}

\newcommand{\cC}{\mathcal{C}}
\newcommand{\cF}{\mathcal{F}}

\newcommand{\cA}{\mathcal{A}}
\newcommand{\aF}{F_{+}}
\newcommand{\gF}{F_{\#}}
\newcommand{\cX}{\mathcal{X}}

\newcommand{\bB}{\mathbf{B}}

\newcommand{\bH}{\mathbf{H}}
\newcommand{\aP}{\mathbf{P}_{+}}
\newcommand{\aD}{\mathbf{D}_{+}}
\newcommand{\gP}{\mathbf{P}_{\#}}
\newcommand{\gD}{\mathbf{D}_{\#}}
\newcommand{\bV}{\mathbf{V}}

  {\par\kern\topsep}
\newcommand\setbar{\ensuremath\ |\ }
\newcommand\forAll{\forall \text{ }}

\newtheorem{MyAssum}{}

\newcommand*{\myQED}{\hfill\ensuremath{\square}}

\renewenvironment{abstract}
{\begin{quote}
\noindent \rule{\linewidth}{.5pt}\par{\bfseries \abstractname.}}
{\medskip\noindent \rule{\linewidth}{.5pt}
\end{quote}
}

\makeatletter
\newtheorem*{rep@theorem}{\rep@title}
\newcommand{\newreptheorem}[2]{%
\newenvironment{rep#1}[1]{%
 \def\rep@title{#2 \ref{##1}}%
 \begin{rep@theorem}}%
 {\end{rep@theorem}}}
\makeatother

\title{Theorems of the Alternative for Conic Integer Programming}

\date{}

\author{\small Temitayo Ajayi}
\author{\small Varun Suriyanarayana}
\author{\small Andrew J. Schaefer}

\affil{\footnotesize Department of Computational and Applied Mathematics, Rice University, Houston, TX 77005, USA}

\onehalfspace

\begin{document}
\maketitle

\begin{abstract}
Farkas' Lemma is a foundational result in linear programming, with implications in duality, optimality conditions, and stochastic and bilevel programming. Its generalizations are known as theorems of the alternative. There exist theorems of the alternative for integer programming and conic programming. We present theorems of the alternative for conic integer programming. We provide a nested procedure to construct a function that characterizes feasibility over right-hand sides and can determine which statement in a theorem of the alternative holds.

\end{abstract}

\keywords{Farkas' lemma, conic integer programming, superadditivity}

\section{Introduction}\label{Introduction}
Assessing the feasibility of optimization problems is important in areas such as bilevel or stochastic optimization, in which knowledge of subproblem feasibility can impact the design and analysis of algorithms. In this paper, we study conic integer programming feasibility. We first provide theorems of the alternative using superadditive duality, then we derive an algorithm to generate a function that represents the feasibility of the conic integer programs over a finite set of right-hand sides. 

We begin this study by reviewing a theorem of the alternative for linear inequalities.

\begin{proposition}\citep{Farkas1894} Let $A \in \mathbb{R}^{m \times n}$ and $b \in \mathbb{R}^{m}.$ Exactly one of the following is true:
\begin{itemize}
\item $\{x \in \mathbb{R}^{n}_{+} \setbar Ax \leq b\} \neq \emptyset$.
\item $\{v \in \mathbb{R}^{m}_{+} \setbar A^{T}v \geq 0, b^{T}v < 0\} \neq \emptyset.$
\end{itemize}
\end{proposition}
This early theorem of the alternative is frequently applied to linear programming. Since the 19$^{th}$ century, there have been many other theorems of the alternative that apply to different optimization problems. A theorem of the alternative for linear Diophantine equations (i.e., integer programming feasibility) can be found in \cite{Schrijver1986}. Other theorems of the alternative for integer programs include \cite{Edmonds1977}, \cite{Blair1982}, \cite{Koppe2004}, \cite{Lasserre2004}, and \cite{Dehghanian2016}. 

Conic programming is a generalization of linear programming that includes well-known topics, such as second-order cone programming \citep{Lobo1998,Alizadeh2003}, semidefinite programming \citep{Vandenberghe1996}, and copositive programming \citep{Dur2010}, among others. A theorem of the alternative for conic programming can be found in \citet[Chapter 6]{Luenberger2015}.

There has been increased interest in connections between linear integer programming and nonlinear integer programming. In particular, connections between discrete optimization and conic programming have yielded new results. \cite{Goemans1995} show that the NP-hard max-cut problem can be well approximated by a semidefinite program, and \cite{Laurent1995} also study semidefinite relaxations of the max-cut problem. The quadratic integer program for max cut (seen in \cite{Goemans1995}) can be reformulated as a semidefinite (conic) integer program with a rank constraint. \cite{Cvetkovic1999} present a binary semidefinite programming formulation for the symmetric traveling salesman problem, and \cite{Manousakis2018} provide a binary semidefinite programming formulation for power system observation placement. \cite{Cezik2005} study Chv\'{a}tal-Gomory cuts in the context of binary conic programming, and \cite{Drewes2009} studies Gomory and Chv\`{a}tal-Gomory cuts for pure second-order cone programming. \cite{Letchford2012} study the convex hulls of binary positive semidefinite matrix feasible regions (also known as spectrahedra). 

Thus, there is interest in both theorems of the alternative and conic integer programming. However, there are no theorems of the alternative for conic integer programming in the literature. Using superadditive functions, we develop the first theorems of the alternative for conic integer programming, and we describe an algorithm that constructs a certificate of infeasibility. 

\section{Preliminaries}
Let $A \in \mathbb{R}^{m \times n}$ and $c \in \mathbb{R}^{n}$. Denote the $j^{th}$ column of $A$ by $a^{j}$ and the $j^{th}$ standard basis vector of $\mathbb{R}^{n}$ by $e_{j}$. Let $K \subset \mathbb{R}^{m}$ be a closed, convex, and pointed cone.  The cone $K$ induces \textit{conic inequalities}; that is, for any $\beta^{1}, \beta^{2} \in \mathbb{R}^{m}, \beta^{1} \preceq_{K} \beta^{2}$ is equivalent to $\beta^{2} - \beta^{1} \in K$. For the special case in which $K$ is full-dimensional, $\beta^{1} \prec_{K} \beta^{2}$ is equivalent to $\beta^{2} - \beta^{1} \in \text{int}(K)$.


Given $\beta \in \mathbb{R}^{m}$, denote the following parametrized conic integer program by CIP($\beta$):
\begin{align*}
\sup \ &c^{T}x\\
\text{s.t.} \ &Ax \preceq_{K} \beta,\\
&x \in \mathbb{Z}^{n}_{+}.
\end{align*}

\begin{definition}\thlabel{superaddDef}
A function $F: \mathbb{R}^{m} \to \mathbb{R}$ is \emph{superadditive} if for any $\beta^{1}, \beta^{2} \in \mathbb{R}^{m},$ $F(\beta^{1}) + F(\beta^{2}) \leq F(\beta^{1} + \beta^{2})$.
\end{definition}

\begin{definition}\thlabel{nondecreasingConeDef}
A function $F: \mathbb{R}^{m} \to \mathbb{R}$ is \emph{nondecreasing with respect to $K$} if $\beta^{1} \preceq_{K} \beta^{2}$ implies $F(\beta^{1}) \leq F(\beta^{2})$.
\end{definition}
Let $\Gamma^{m} \coloneqq \{F: \mathbb{R}^{m} \to \mathbb{R} \setbar F \text{ is superadditive and nondecreasing with respect to } K\}.$

\cite{Moran2018} present a superadditive dual for conic mixed-integer programs with free variables. The following formulation is adapted for conic integer programs with nonnegative variables. 
\begin{subequations}\label{MoranDual}
\begin{align}
\inf \ & F(\beta)\\
\text{s.t.} \ &F(a^{j}) \geq c_{j}, \forAll j \in \{1,...,n\},\\
&F(0) = 0,\\
&F \in \Gamma^{m}.
\end{align}
\end{subequations}

\begin{proposition} \thlabel{ourWeakDuality}
Let $x$ be a feasible solution to CIP($\beta$), and let $F$ be a feasible solution to \eqref{MoranDual}. Then $F(\beta) \geq c^{T}x$.
\end{proposition}

We note that \cite{Moran2018} state a similar result when $K$ is a regular cone (closed, pointed, convex, and full-dimensional). In fact, \thref{ourWeakDuality} holds without full-dimensionality. 

\begin{proposition}\thlabel{MoranIPDual}\citep{Moran2018} Consider CIP($\beta$), where $K$ is a regular cone. Then \eqref{MoranDual} is a strong dual to CIP($\beta$). 
\end{proposition}

Proofs of \thref{ourWeakDuality,MoranIPDual} are in the appendix. Earlier statements of \thref{MoranIPDual}, with less general conditions, exist in the literature (e.g., \cite{Moran12}). We remark that \cite{Moran2018} assume that both the primal and dual are feasible for their strong duality result. Hence, a theorem of the alternative for feasibility is not immediate from the dual. 

\section{Theorems of the Alternative for Conic Integer Programs}

In this section, we provide the first theorems of the alternative for conic integer programs. \cite{Dehghanian2016} derive theorems of the alternative for linear integer programs in which either the integer program is feasible, or they construct a superadditive function that certifies infeasibility. Hence, we establish theorems of the alternative for conic integer programs by generalizing the approach of \cite{Dehghanian2016}. That is, we also characterize the feasibility problem as a function of the right-hand side, and develop a method to compute this function. We note, importantly, that our method does not generally apply to non-integral right-hand side vectors. Some of our results apply when the cone $K$ is not full-dimensional, as they only rely on weak duality.
\begin{definition}\thlabel{def_aSets}
For any $\beta \in \mathbb{R}^{m}$, define the sets $\aP(\beta)$ and $\aD(\beta)$ as follows:
\begin{itemize}
\item $\aP(\beta) \coloneqq \{x \setbar Ax \preceq_{K} \beta, x \in \mathbb{Z}^{n}_{+}\}.$
\item $\aD(\beta) \coloneqq \{F \in \Gamma^{m} \setbar F(\beta) < 0, F(a^{j}) \geq 0\}$.
\end{itemize}
\end{definition}
\begin{definition}\thlabel{def_aF}
Define the \textit{feasibility function} $\aF: \mathbb{R}^{m} \to \mathbb{R}$ as follows:
\begin{equation}\label{aF}
\aF(\beta) \coloneqq
\begin{cases*}
0 & if $ \aP(\beta) \neq \emptyset, $ \\
-1 & if $\aP(\beta) = \emptyset$.
\end{cases*}
\end{equation}
\end{definition}

\begin{proposition} \thlabel{schaefer5 extension}
The function $\aF$ is superadditive and nondecreasing with respect to $K$, that is, $\aF \in \Gamma^{m}$.
\end{proposition}

\begin{proof} \textbf{Superadditivity:} 
Suppose not. Because $\aF$ takes values of 0 and $-1$ only, it is necessary that there exist $\beta^{1}, \beta^{2} \in \mathbb{R}^{m}$ such that $\aF(\beta^{1} + \beta^{2})=-1$ and $\aF(\beta^{1})=\aF(\beta^{2})=0$. However, in this case there exist $x^{1}, x^{2} \in \mathbb{Z}^{n}_{+}$ such that $Ax^{1} \preceq_K \beta^{1},$ and $Ax^{2} \preceq_K \beta^{2}.$
Notice that $A(x^{1} + x^{2}) \preceq_K (\beta^{1} + \beta^{2}),$ and $(x^{1} + x^{2}) \in \mathbb{Z}^{n}_{+},$ which imply $\aF(\beta^{1} + \beta^{2}) = 0$, a contradiction. Hence $\aF$ is superadditive.

\textbf{Nondecreasing:} Let $\beta^{1}, \beta^{2} \in \mathbb{R}^{m}$ such that $\beta^{1} \succeq_{K} \beta^{2}$. If $\aF(\beta^{2}) = -1$, then we are done, so assume that $\aF(\beta^{2}) = 0$. Thus, there exists $x \in \mathbb{Z}^{n}_{+}$ such that $0 \preceq_{K} \beta^{2} - Ax$. Further, $\beta^{1} \succeq_{K} \beta^{2}$ implies that $\beta^{1} - Ax \succeq_{K} \beta^{2} - Ax \succeq_{K} 0$. Hence, $\aF(\beta^{1}) = 0$ and $\aF$ is nondecreasing. 
\end{proof}

\thref{schaefer5 extension} shows that the feasibility problem for conic integer programs over right-hand sides can be characterized by a superadditive and nondecreasing (with respect to $K$) function. Next, we use the function $\aF$ to prove theorems of the alternative for conic integer programs.

\begin{proposition} \thlabel{schaefer1 extension}
For all $\beta \in \mathbb{R}^{m}$, exactly one of the following holds:
\begin{itemize}
\item $\aP(\beta) \neq \emptyset$. \label{schaefer1Extension1}
\item $\aD(\beta) \neq \emptyset$. \label{schaefer1Extension2}
\end{itemize}
\end{proposition}

\begin{proof}
First, suppose that  $\aP(\beta) \neq \emptyset$, and further suppose that $\aD(\beta) \neq \emptyset$. Consider the primal conic integer program \eqref{prop1P} and its superadditive dual \eqref{prop1D}:
\begin{align}
&\sup\{0^{T}x \setbar Ax \preceq_{K} \beta, x \in \mathbb{Z}^{n}_{+}\}, \label{prop1P}\\
&\inf\{F(\beta) \setbar F(a^{j}) \geq 0, \forAll j = 1,...,n, F \in \Gamma^{m}\}. \label{prop1D}
\end{align}

Let $\hat{x}$ be feasible for \eqref{prop1P}, and let $\widehat{F}$ be feasible for \eqref{prop1D} such that $\widehat{F}(\beta) < 0$, which exists by assumption. By weak duality (\thref{ourWeakDuality}), $0 = 0^{T}\hat{x} \leq \widehat{F}(\beta) < 0$, a contradiction. Hence, if $\aP(\beta) \neq \emptyset$, $\aD(\beta) = \emptyset$. For the alternative case ($\aP(\beta) = \emptyset$), consider $\aF$ as defined in \eqref{aF}. By \thref{schaefer5 extension}, $\aF \in \Gamma^{m}$. Because $\aP(\beta) = \emptyset$, $\aF(\beta) = -1 < 0$. Additionally, $\aF(a^{j}) \geq 0$ for all $j \in \{1,...,n\}$ because $e_{j} \in \{x \in \mathbb{Z}^{n}_{+} \setbar Ax \preceq_{K} a^{j}\}$. Thus, $\aD(\beta) \neq \emptyset$.
\end{proof}
\begin{corollary}\thlabel{acuteFCorollary}
For all $\beta \in \mathbb{R}^{m}$, the following three statements are equivalent.
\begin{itemize}
\item The set $\aD(\beta)$ is nonempty.
\item $\aF \in \aD(\beta)$. 
\item $\aF(\beta) = -1$. 
\end{itemize}
\end{corollary}

\thref{acuteFCorollary} implies that one can determine if $\aD(\beta)$ is nonempty for a range of different right-hand sides by constructing $\aF$ and checking if $\aF(\beta) = -1$ for the right-hand side $\beta$ of interest. Moreover, each of the equivalent statements in \thref{acuteFCorollary} is a valid second statement in the theorem of the alternative, \thref{schaefer1 extension}. 

Similar results hold for conic integer programming with free variables.
\begin{definition}\thlabel{def_gSets}
For any $\beta \in \mathbb{R}^{m}$, define the sets $\gP(\beta)$ and $\gD(\beta)$ as follows:
\begin{itemize}
\item $\gP(\beta) \coloneqq \{x \setbar Ax \preceq_{K} \beta, x \in \mathbb{Z}^{n}\}.$
\item $\gD(\beta) \coloneqq \{F \in \Gamma^{m} \setbar F(\beta) < 0, F(a^{j}) = 0\}$.
\end{itemize}
\end{definition}
\begin{definition}\thlabel{def_gF}
Define the \textit{feasibility function} $\gF: \mathbb{R}^{m} \to \mathbb{R}$ as follows:
\begin{equation}\label{gF}
\gF(\beta) \coloneqq
\begin{cases*}
0 & if $ \gP(\beta) \neq \emptyset, $ \\
-1 & if $\gP(\beta) = \emptyset$.
\end{cases*}
\end{equation}
\end{definition}



\begin{proposition} 
\thlabel{schaefer3 extension}
For all $\beta \in \mathbb{R}^{m}$, exactly one of the following holds:
\begin{itemize}
\item $\gP(\beta) \neq \emptyset$.\label{Schaefer3Extension1}
\item $\gD(\beta) \neq \emptyset$. \label{Schaefer3Extension2}
\end{itemize}
\end{proposition}

\begin{proof}
Consider the following primal conic integer program \eqref{prop2P} and its superadditve dual \eqref{prop2D}:
\begin{align}
&\sup\{0^{T}x \setbar Ax \preceq_{K} \beta, x \in \mathbb{Z}^{n}\}, \label{prop2P}\\
&\inf\{F(\beta) \setbar F(a^{j}) = 0, \forAll j = 1,...,n, F \in \Gamma^{m}\}. \label{prop2D}
\end{align}
Suppose $\gP(\beta) \neq \emptyset$ and let $\hat{x}$ be a feasible solution to \eqref{prop2P}. For any dual feasible $\widehat{F}$, $\widehat{F}(\beta) \geq \widehat{F}(A\hat{x})$ because $A\hat{x} \preceq_{K} \beta$ and $\widehat{F}$ is nondecreasing with respect to $K$. By superadditivity, $\widehat{F}(A\hat{x}) \geq \sum\limits_{j = 1}^{n} \widehat{F}(a^{j}\hat{x}_{j}) \geq \sum\limits_{j = 1}^{n}\widehat{F}(a^{j})\hat{x}_{j}.$ Because $\widehat{F}$ is a feasible dual solution, $\widehat{F}(a^{j}) = 0,$ for all $j \in \{1,...,n\}$; hence, $\widehat{F}(\beta) \geq \sum\limits_{j = 1}^{n}\widehat{F}(a^{j})\hat{x}_{j} = 0$. This shows that $\widehat{F}(\beta) \geq 0$ and that $\gD(\beta) = \emptyset$.

Now suppose that $\gP(\beta) = \emptyset$, and consider the function $\gF.$ Notice that for all $j \in \{1,...,n\},$ $\gF(a^{j})=0$ because $Ae_j=a^{j}\preceq_K a^{j},$ and $e_j\in \mathbb{Z}^n_{+}$. Also, because $\gP(\beta) = \emptyset$, $\gF(\beta) = -1$. The rest of the proof is similar to that of \thref{schaefer5 extension}.
\end{proof}

\begin{corollary}\thlabel{graveFCorollary}
For all $\beta \in \mathbb{R}^{m}$, the following three statements are equivalent.
\begin{itemize}
\item The set $\gD(\beta)$ is nonempty.
\item $\gF \in \gD(\beta)$. 
\item $\gF(\beta) = -1$. 
\end{itemize}
\end{corollary}

Similar to the nonnegative case, each of the equivalent statements in \thref{graveFCorollary} is a valid second statement in the theorem of the alternative, \thref{schaefer3 extension}.

\section{An Algorithm to Construct $\aF$} \label{constructF}
We construct a certificate of infeasibility for the conic integer program. A conic integer program with $n$ free variables can be converted into a conic integer program with $2n$ nonnegative variables.
Hence, for the remainder of this paper, we focus only on conic integer programs with nonnegative variables, and our algorithm constructs $\aF$. 
Although $\aF$ characterizes the feasibility problem over all of $\mathbb{R}^{m}$, it may be prohibitive to evaluate directly. Therefore, it is beneficial to infer the value of $\aF$ at a certain right-hand side by using previously computed values of $\aF$.

\begin{MyAssum}\thlabel{AisIntegral}
The constraint matrix $A$ is integral.
\end{MyAssum}

In linear and integer programming, the constraint data is often assumed to be rational. Therefore,\thref{AisIntegral} is largely for convenience.

\begin{proposition} \thlabel{schaefer6 extension}
If $K$ has nonempty interior, then for any $\beta \in \mathbb{R}^{m}$, there exists a $\beta'\in \mathbb{Z}^m$ such that $\beta - \beta'\in K$ and $\aF(\beta)=\aF(\beta').$
\end{proposition}

\begin{proof} First, suppose $\aF(\beta) = -1$. Let $B = \{\beta - v \setbar v \in K\}$. By hypothesis, $K$ has a nonempty interior, and because $B$ is a translation of $K$, $B$ also has a nonempty interior. Therefore, there exists a rational vector in $B$, which implies that $B$ contains an integral point $\tilde{\beta} = \beta - \beta'$, for some $\beta' \in K$. Thus, $\beta - \beta' \in K$. Because $\aF$ is nondecreasing, $-1 = \aF(\beta) \geq \aF(\beta') \geq -1,$ thus, $\aF(\beta') = -1 = \aF(\beta) = \aF(\beta')$.

Now consider $\aF(\beta) = 0$, which implies the existence of $\hat{x} \in \aP(\beta)$. The matrix $A$ and solution $\hat{x}$ are integral, so $A\hat{x}$ is integral. Trivially, $A\hat{x} \preceq_{K} A\hat{x}$, which implies $\aF(A\hat{x}) = \aF(\beta) = 0$.
\end{proof}

We note that \thref{schaefer6 extension} is similar to \citet[Proposition 6]{Dehghanian2016}, where $\beta'$ is analogous to $\lfloor \beta \rfloor$. The main difference, however, is that \thref{schaefer6 extension} only asserts the existence of such a $\beta'$. In contrast, for \citet[Proposition 6]{Dehghanian2016}, $\lfloor \beta \rfloor$ is easily computed. Consequently, for the remainder of this paper, we focus only on the feasibility problem over integral right-hand sides. The sequel provides a nested approach to compute $\aF(\beta)$ for all $\beta \in \mathbb{Z}^{m}$.

\begin{definition}\thlabel{def_aSetsk}
For any $\beta \in \mathbb{R}^{m}, k \in \mathbb{Z}_{+}$, define the set $\aP^{k}(\beta)$ as 
\begin{align*}
\aP^{k}(\beta) \coloneqq \{x \setbar Ax \preceq_{K} \beta, x \in \mathbb{Z}^{n}_{+}, 1^{T}x \leq k\}.
\end{align*}
\end{definition}

\begin{definition}\thlabel{def_aFk}
For all $k \in \mathbb{Z}_{+}$, define:
\begin{equation}
F^k(\beta) \coloneqq
\begin{cases*}
0 & if $ \aP^{k} \neq \emptyset, $ \\
-1 & if $\aP^{k} = \emptyset$.
\end{cases*}
\end{equation}
\end{definition}

\thref{Schaefer7 extension} proves a nested property for a sequence of functions that characterize a 1-norm constrained feasibility problem, and \thref{Schaefer8 extension} proves that this sequence of functions converges to $\aF$.

\begin{proposition} \thlabel{Schaefer7 extension}
For $k\geq 1,  \aF^k(\beta)=\max\{\aF^{k-1}(\beta), \max\limits_{j \in \{1,...n\}} \aF^{k-1}(\beta-a^{j})\}$.
\end{proposition}

\begin{proof} For any $\beta \in \mathbb{Z}^{m}$ and $k \geq 1$, it is clear that $\aF^{k-1}(\beta) \leq \aF^{k}(\beta)$ as $\aP^{k-1}(\beta) \subseteq \aP^{k}(\beta)$.

Suppose $\aF^{k-1}(\beta) = 0$. Then $0 = \aF^{k-1}(\beta) \leq \aF^{k}(\beta) \leq 0$, so $\aF^{k-1}(\beta) = \aF^{k}(\beta) = 0 = \max\{\aF^{k-1}(\beta), \max\limits_{j \in \{1,...,n\}} \aF^{k-1}(\beta - a^{j})\}$. 

Now suppose that $\aF^{k-1}(\beta) = -1,$ and further suppose that $\aF^{k-1}(\beta - a^{j}) = -1$ for all $j \in \{1,...,n\}$. Then $\max\{\aF^{k-1}(\beta), \max\limits_{j \in \{1,...,n\}} \aF^{k-1}(\beta - a^{j})\} = -1.$ Because $\aF^{k-1}(\beta) = -1,$ for any $x \in \mathbb{Z}^{n}_{+}$ such that $Ax \preceq_{K} \beta, 1^{T}x \leq k,$ we also have that $1^{T}x = k$. Thus, there exists $j^{*} \in \{1,...,n\}$ such that $x_{j^{*}} > 0$. Observe that $Ax = A(x - e_{j^{*}}) + a^{j^{*}} \preceq_{K} \beta \iff A(x - e_{j^{*}}) \preceq_{K} \beta - a^{j^{*}}$. Further, $x - e_{j^{*}} \in \mathbb{Z}^{n}_{+},$ and $1^{T}(x - e_{j^{*}}) = k - 1$, which would imply $\aF^{k-1}(\beta - a^{j^{*}}) = 0$, a contradiction. Thus, $\aP^{k}(\beta) = \emptyset$ and $\aF^{k}(\beta) = -1 = \max\{\aF^{k-1}(\beta), \max\limits_{j \in \{1,...,n\}} \aF^{k-1}(\beta - a^{j})\}$.

Suppose that $\aF^{k-1}(\beta) = -1$ and further suppose that there exists $j^{*}$ such that $\aF^{k-1}(\beta - a^{j^{*}}) = 0$. Then there exists $x \in \mathbb{Z}^{n}_{+}$ such that $Ax \preceq_{K} \beta - a^{j^{*}},$ with $1^{T}x \leq k - 1$. It follows that $x + e_{j^{*}} \in \mathbb{Z}^{n}_{+}, A(x + e_{j^{*}}) = Ax + a^{j^{*}} \preceq_{K} \beta - a^{j^{*}} + a^{j^{*}} = \beta$, and $1^{T}(x + e_{j^{*}}) \leq k$. Thus, $\aF^{k}(\beta) = 0 = \max\{\aF^{k-1}(\beta), \max\limits_{j \in \{1,...,n\}} \aF^{k-1}(\beta - a^{j})\}$.
\end{proof}

\begin{proposition}\thlabel{Schaefer8 extension}
The sequence of functions $\{\aF^{k}\}_{k = 0}^{\infty}$ converges pointwise to $\aF$. 
\end{proposition}

\begin{proof} 
Suppose that $\aF(\beta) = 0$. Then there exists $x \in \mathbb{Z}^{n}_{+}$ such that $Ax \preceq_{K} \beta$. Let $l = 1^{T}x$, then $\aF^{l'}(\beta) = 0$ for all $l' \geq l$, and $\lim\limits_{k \to \infty} \aF^{k}(\beta) = 0 = \aF(\beta)$.

Now suppose that $\aF(\beta) = -1$. Then $\emptyset = \aP(\beta) \supseteq \aP^{k}(\beta),$ for all $k \in \mathbb{Z}_{+}$. Thus, $\aF^{k}(\beta) = -1$ for all $k \in \mathbb{Z}_{+}$ and $\lim\limits_{k \to \infty} \aF^{k}(\beta) = -1 = \aF(\beta)$.  
\end{proof}

Define $\bH$ to be the set of right-hand sides over which we wish to determine feasibility, which we assume is a bounded subset of $\mathbb{Z}^{m}$. Algorithm \ref{F_alg_short} briefly summarizes the procedure that is explained throughout this section to construct $\aF^{\bar{k}}$, for a given $\bar{k} \in \mathbb{Z}_{+}$, over $\bH$.  \thref{Schaefer13 extension} computes a finite integer $\bar{k}$ such that $\aF^{\bar{k}} = \aF$; thus, Algorithm \ref{F_alg_short} can compute $\aF$ over $\bH$. Additional details, including a more comprehensive description of the algorithm, are in the appendix.

\begin{algorithm}
\caption{Construction of feasibility certificate $\aF^{\bar{k}}$}\label{F_alg_short}
\begin{algorithmic}[1]
\Procedure{MAIN}{}
\State Given: $A, \bar{k}, \bH, K$
\For{$k = 1,...,\bar{k}$}
    \For{$\beta \in \mathbf{H}$}
        \State $\aF^{k}(\beta) = \max\{\aF^{k-1}(\beta), \max\limits_{j \in \{1,...,n\}} \aF^{k-1}(\beta - a^{j})\}$
    \EndFor
\EndFor
\EndProcedure
\end{algorithmic}
\end{algorithm}

We remark that the method to evaluate $\aF^{k}(\beta) = \max\{\aF^{k-1}(\beta), \max\limits_{j \in \{1,...,n\}} \aF^{k-1}(\beta - a^{j})\}$ depends on properties of $\beta$. For instance, if $\aF^{k-1}(\beta) = 0$, then $\aF^{k}(\beta) = 0$, and no additional computation is needed. In other cases, we use the level-set-minimal vectors of $\aF^{k-1}$ in the evaluation (see \thref{levelSetMinDef}). We show later in this section that the level-set-minimal vectors necessary to construct $\aF^{k}$ can be computed within the algorithm.




We denote the ``extended integers" by $\bar{\mathbb{Z}} \coloneqq \mathbb{Z} \cup \{-\infty,\infty\}.$ 
\begin{definition}\thlabel{levelSetMinDef}\text{ }
\begin{itemize}
\item Define $\bV \coloneqq (\mathbb{Z}^{m} \cap K) \backslash \{0\}.$
\item For each $k\in \mathbb{Z}_{+}$, define $\bB^k \coloneqq \{\beta \in \mathbb{Z}^{m} : \aF^k(\beta)=0, \aF^k(\beta-v)=-1 \text{ for all } v\in \bV\}.$
\item Define ${\bB} \coloneqq \{\beta \in \bar{\mathbb{Z}}^m: \aF(\beta)=0, \aF(\beta-v)=-1 \text{ for all } v \in \bV\}.$
\end{itemize}
\end{definition}

The sets ${\bB}$ and $\bB^{k}$ for $k \in \mathbb{Z}_{+}$ describe a notion of level-set-minimal vectors for $\aF$ and $\aF^{k}, k \in \mathbb{Z}_{+}$, respectively.  \cite{Trapp2013} use level-set-minimal vectors for integer programming value functions with respect to $\mathbb{R}^{m}_{+}$; whereas ${\bB}$ and $\bB^{k}$ are with respect to the cone $K$. In our algorithm, computing $\bB^{k}$ enables one to compute $\beta$ implicitly by attempting to find $\bar{\beta} \in \bB^{k}$ such that $\bar{\beta} \preceq_{K} \beta$. A key difference between our work and that of \cite{Dehghanian2016} is that, using the definitions of ${\bB}$ (respectively $\bB^{k}$) directly, one must check that a vector $\beta$ is minimal by evaluating $\aF$ (respectively $\aF^{k}$) at infinitely many points for the general conic case. The linear integer programming case studied by \cite{Dehghanian2016} only requires $m$ such evaluations. Therefore, the nested procedure that follows is even more vital in the conic setting.

\begin{remark}\thlabel{remarkbB}\text{ }
\begin{itemize}
\item For any $\beta \in \bV$, $\beta \not\in \bB^{k},$ for all $k \in \mathbb{Z}_{+}$. This follows from the fact that $\aF^{k}(\beta - \beta) = \aF^{k}(0) = 0$ for all $k \in \mathbb{Z}_{+}$. 
\item For any $\beta$, $\aP^{0}(\beta) = \{0\}$. Thus, for any $\beta \in \mathbb{Z}^{m}$ such that $\aF^{0}(\beta) = 0, \beta \in K$. Also, for any $\beta \in \bV$, $\aF^{0}(\beta - \beta) = 0$, which implies that $\bB^{0} = \{0\}$.
\end{itemize}
\end{remark}

\thref{Schaefer9 extension} determines $\aF^{k}(\beta)$ using the sets $\bB^{k}$.

\begin{proposition} \thlabel{Schaefer9 extension}
For each $\beta \in \mathbb{Z}^m, k\in \mathbb{Z}_+$, $\aF^k(\beta)=0$ if and only if there exists a $\bar{\beta} \in \bB^{k}$ such that $\bar{\beta}\preceq_K \beta$. 
\end{proposition}
\begin{proof} 
Fix $k \in \mathbb{Z}^{n}_{+}$. For each $x \in \mathbb{Z}^{n}_{+}$ such that $1^{T}x \leq k$, let $\mathbf{Q}_{x}^{k} = \{v \in \mathbb{R}^{m} \setbar Ax \preceq_{K} v\}$. Also, define $\cX^{k} = \{x \in \mathbb{Z}^{n}_{+} \setbar 1^{T}x \leq k, \beta \in \mathbf{Q}_{x}^{k}\},$ and $\cA^{k} = \{Ax \setbar x \in \cX^{k}\}$. Notice that $\aF^{k}(\beta) = 0$ if and only if  $\cX^{k} \neq \emptyset$, and thus, $\aF^{k}(\beta) = 0$ if and only if $\cA^{k} \neq \emptyset$. In addition, $n < \infty,$ implies that $\cX^{k}$ is a finite set, which also implies that $\cA^{k}$ is a finite set. 

Because $\preceq_{K}$ defines a partial order on $\mathbb{R}^{m}$, $\aF^{k}(\beta) = 0$ if and only if there exists $A\bar{x}$ with $\bar{x} \in \cX^{k}$ such that: for all $z \in \cX^{k}$, $Az \preceq_{K} A\bar{x}$ if and only if $Az = A\bar{x}$. Trivially, $\aF^{k}(A\bar{x}) = 0$ and $A\bar{x} \in \bB^{k}$.

We now prove the statement. Suppose that $\aF^{k}(\beta) = 0$. Consider any $v \in \bV$. Suppose that $\aF^{k}(A\bar{x} - v) = 0$ with $\bar{x}$ as defined above. Then there exists $\tilde{x} \in \mathbb{Z}^{n}_{+}$ such that $A\tilde{x} \preceq_{K} A\bar{x} - v,$ and $1^{T}\tilde{x} \leq k$. Thus, $A\tilde{x} \preceq_{K} A\bar{x}$ but $A\tilde{x} \neq A\bar{x}$, a contradiction. Hence, $\aF^{k}(\beta) = 0$ implies the existence of $\bar{\beta} \in \bB^{k}$---namely, $\bar{\beta} = A\bar{x}$---such that $\bar{\beta} \preceq_{K} \beta$.

To prove the other direction, suppose that there exists $\bar{\beta} \in \bB^{k}$ such that $\bar{\beta} \preceq_{K} \beta$. Then there exists $\hat{x} \in \mathbb{Z}^{n}_{+}$ such that $A\hat{x} \preceq_{K} \bar{\beta}$, $1^{T}\hat{x} \leq k,$ and $A\hat{x} = \bar{\beta}$ because $\aF^{k}(\bar{\beta} - v) = -1$, for all $v \in \bV$. Thus $A\hat{x} = \bar{\beta} \preceq_{K} \beta$ so $\hat{x} \in \cX^{k}$ such that $Az \preceq_{K} A\hat{x}$ if and only if $Az = A\hat{x}$, for all $z \in \cX^{k}$. 
%
\end{proof}

\begin{proposition} \thlabel{Schaefer10 extension} \text{ }
\begin{enumerate}[series = schaef10, label = (\ref{Schaefer10 extension}.\textbf{\alph*})]
\item For each $\beta \in \mathbb{Z}^m$, $\beta\in \bB^k$ if and only if both $\aF^k(\beta)=0$ and for each $\bar{x}\in \aP^{k}(\beta)$, $A\bar{x}=\beta$. \label{Schaefer10extension1}
\item For each $\beta \in \mathbb{Z}^m, \beta \in \boldmath{{\bB}}$ if and only if both $\aF(\beta)=0$ and for each $\bar{x}\in \aP(\beta)$, $A\bar{x}=\beta$. \label{Schaefer10extension2}
\end{enumerate}
\end{proposition}
\proof 
(\ref{Schaefer10extension1} $\Rightarrow$):
Because $\beta \in \bB^k$, $\aF^k(\beta)=0$. Suppose there exists an $\bar{x}\in \aP^{k}(\beta)$ such that $A\bar{x}\neq \beta$. Then $\beta-A\bar{x}\in K,\beta-A\bar{x}\neq 0, \bar{x} \in \mathbb{Z}^n_+$. However, if we set $\bar{\beta}=A\bar{x}$, it is clear that $\bar{\beta}\preceq_K \beta$, and because $A$ and $\bar{x}$ are integral, $\bar{\beta} \in \mathbb{Z}^m$. Moreover, $\aF^k(\bar{\beta})=0$ because $\bar{x}$ is a solution to CIP($\bar{\beta}$). However, by the definition of $\bB^k$, this implies that $\beta \not \in \bB^k$. This is a contradiction. 

(\ref{Schaefer10extension1} $\Leftarrow$):
Suppose that $\aF^k(\beta)=0$ and for each $\bar{x}\in \aP^{k}(\beta)$, $A\bar{x}=\beta$ but $\beta \not \in \bB^k$. By the definition of $\bB^k$, there exists an integral $\bar{\beta} \neq \beta$ such that $\bar{\beta} \preceq_K \beta$ and $\aF^k(\beta)=0$. This implies that there exists an $\bar{x}$ satisfying $A\bar{x} \preceq_K \bar{\beta} \preceq_K \beta, \boldmath{1}^T\bar{x}\leq k, \bar{x}\in \mathbb{Z}^n_+$. Because $A\bar{x}= \beta$, $\beta \preceq_K\bar{\beta}$, thus, $\beta=\bar{\beta}$. However, this is a contradiction, which implies the reverse implication holds as well.

The proof of \ref{Schaefer10extension2} is similar. \qed

\begin{corollary} \thlabel{Schaefer10 extension corollary}
For each $k\in \mathbb{Z}_+$, $\bB^k$ is finite.
\end{corollary}
\begin{proof} Observe that because $|\{x \in \mathbb{Z}^{n}_{+} \setbar 1^{T}x \leq k\}| < \infty$, we also have $|\{Ax \setbar x \in \mathbb{Z}^{n}_{+}, 1^{T}x \leq k\}| < \infty$. 
Let $\beta \in \mathbf{B}^{k},$ then $\aF^{k}(\beta) = 0$ and $\{x \in \mathbb{Z}^{n}_{+} \setbar 1^{T}x \leq k\} \neq \emptyset$. By \thref{Schaefer10 extension}, if $\bar{x} \in \{x \in \mathbb{Z}^{n}_{+} \setbar 1^{T}x \leq k\}$, then $A\bar{x} = \beta$. Hence, $\mathbf{B}^{k} \subseteq \{Ax \setbar x \in \mathbb{Z}^{n}_{+}, 1^{T}x \leq k\}$, and so $\mathbf{B}^{k}$ is finite. 
\end{proof}

The fact that the sets $\bB^{k}$ are finite implies that one may search through the level-set-minimal vectors as part of a finite algorithm. However, one must still construct the sets $\bB^{k}$; as stated earlier, verifying that $\beta \in \bB^{k}$ directly can require determining if $\aF^{k}(\beta - v) = -1$ for all $v$ in the set $\mathbf{V}$, which is countably infinite in general. Thus, we also construct $\bB^{k}$ through a finite nested procedure.

For any $k \in \mathbb{Z}_{+}, j \in \{1,...,n\}$, let $\bB^{k} + a^{j} \coloneqq \{\beta \in \mathbb{Z}^{m} \setbar \beta - a^{j} \in \bB^{k}\}$. If $\beta - a^{j} \in \bB^{k}$, or if $\beta \in \bB^{k}$, then $\beta$ may belong to $\bB^{k+1}$; \thref{def:Ck,Schaefer11 extension} address this notion formally.

\begin{definition}\thlabel{def:Ck}
For each $k \in \mathbb{Z}_{+}$, define $\mathbf{C}^{k} \coloneqq \{\beta \in \bB^{k} \bigcup\limits_{j = 1}^{n} (\bB^{k} + a^{j}) \setbar \aF^{k}(\beta) = 0 \Rightarrow \beta \in \bB^{k} \text{ and } \aF^{k}(\beta - a^{\ell}) = 0 \Rightarrow \beta - a^{\ell} \in \bB^{k}, \forAll \ell \in \{1,...,n\}\}.$
\end{definition}

\thref{Schaefer11 extension} states that one can identify the set of level-set-minimal vectors $\bB^{k}$ with a nested procedure. The construction of the sets $\mathbf{C}^{k}$ constitutes an intermediate step in this procedure.

\begin{proposition} \thlabel{Schaefer11 extension}
For each integer $k\geq 1$, $\mathbf{B}^{k}=\{\beta \in \mathbf{C}^{k-1} \setbar \aF^k(\beta)=0\}.$
\end{proposition}
\proof ($\subseteq$): 
Suppose $\beta' \in \bB^{k}$. It is immediate that $\aF^{k}(\beta') = 0$. Hence, we first show that $\aF^{k-1}(\beta') = 0 \Rightarrow \beta' \in \bB^{k-1} \text{ and } \aF^{k-1}(\beta' - a^{j}) = 0 \Rightarrow \beta' - a^{j} \in \bB^{k-1}, \forAll j \in \{1,...,n\},$ and then we show that $\beta' \in \bB^{k-1} \bigcup\limits_{j = 1}^{n} (\bB^{k-1} + a^{j})$, which together imply that $\beta \in \mathbf{C}^{k-1}$.

Because $\beta' \in \bB^{k}$, $\aF^{k}(\beta' - v) = -1,$ for all $v \in \bV$. By \thref{Schaefer7 extension}, $\aF^{k}(\beta' - v) \geq \aF^{k-1}(\beta' - v - a^{j})$, for all $j \in \{1,...,n\}$. Thus, $\aF^{k-1}(\beta' - v - a^{j}) = -1$, for all $v \in \bV$, $j \in \{1,...,n\}$. This implies that for any $j$ such that $\aF^{k-1}(\beta' - a^{j}) = 0, \beta' - a^{j} \in \bB^{k-1}$. Also, because $\beta' \in \bB^{k}$, by \thref{BkRelationship}, if $\aF^{k-1}(\beta') = 0$, then $\beta' \in \bB^{k-1}$.

We now show that $\beta' \in \bB^{k-1} \bigcup\limits_{j = 1}^{n}(\bB^{k-1} + a^{j})$. Because $\beta' \in \bB^{k},$ we have $\aF^{k}(\beta') = 0$, which implies by \thref{Schaefer7 extension}, at least one of the following holds: $\aF^{k-1}(\beta') = 0$, or $\aF^{k-1}(\beta' - a^{j}) = 0$ for some $j \in \{1,...n\}$. As shown above, if $\aF^{k-1}(\beta') = 0,$ then $\beta' \in \bB^{k-1}$, and if, for some $j \in \{1,...,n\},$ $\aF^{k-1}(\beta' - a^{j}) = 0$, then $\beta' - a^{j} \in \bB^{k-1}$. Thus, $\beta' \in \bB^{k-1} \bigcup\limits_{j = 1}^{n} (\bB^{k-1} + a^{j})$.
%
%
%
Moreover, by the definition of $\bB^{k}$, $\aF^{k}(\beta') = 0$. Hence, $\beta' \in \bB^{k}$ implies that $\beta' \in \{\beta \in \mathbf{C}^{k-1} \setbar \aF^{k}(\beta) = 0\}$.

\noindent ($\supseteq$): Suppose there exists $\beta' \in \mathbf{C}^{k-1}$ such that $\aF^{k}(\beta') = 0$. 

Consider $j \in \{1,...,n\}$ and suppose that $\aF^{k-1}(\beta' - a^{j}) = 0$. Then $\aF^{k-1}(\beta' - a^{j} - v) = -1$, for all $v \in \bV$ because $\beta' - a^{j} \in \bB^{k-1}$ (due to $\beta' \in \mathbf{C}^{k-1})$. If instead, $\aF^{k - 1}(\beta' - a^{j}) = -1,$ then $-1 = \aF^{k - 1}(\beta' - a^{j}) \geq \aF^{k-1}(\beta' - a^{j} - v) \geq -1$. Thus, $\aF^{k-1}(\beta' - a^{j} - v) = -1$. Suppose $\aF^{k-1}(\beta') = -1$; the monotonicity of $\aF^{k-1}$ implies $\aF^{k-1}(\beta' - v) = -1$, for all $v \in \bV$. On the other hand, suppose $\aF^{k-1}(\beta') = 0$, then $\beta' \in \bB^{k-1}$ (because $\beta' \in \mathbf{C}^{k-1}$), and this implies $\aF^{k-1}(\beta' - v) = -1$, for all $v \in \bV$. Thus, $\aF^{k}(\beta' - v) = \max\{\aF^{k-1}(\beta' - v), \max\limits_{j \in \{1,...,n\}} \aF^{k-1}(\beta' - v - a^{j})\} = -1$, for all $v \in \bV$. Therefore, $\beta' \in \bB^{k}$.
\qed

\thref{BkRelationship} proves a relationship between the sets $\bB^{k}$ and $\bB^{l}$ for any $k, l \in \mathbb{Z}_{+}$. This result is useful when determining the level-set-minimal vectors at each iteration.
\begin{proposition}\thlabel{BkRelationship}
For any $k \leq l$, if $\beta \not\in \bB^{k}$ and $\aF^{k}(\beta) = 0$, then $\beta \not\in \bB^{l}$.
\end{proposition}
\begin{proof}
%
Because $\aF^{k}(\beta) = 0$ and $\beta \not\in \bB^{k}$, there exists $v \in \bV$ such that $\aF^{k}(\beta - v) = 0$. Because $\aP^{k}(\beta - v) \subseteq \aP^{l}(\beta - v),$ we have $\aF^{l}(\beta - v) = 0$. This implies $\beta \not\in \bB^{l}$.
\end{proof}

Denote the dual cone of $K$ by $K^{*} = \{h \in \mathbb{R}^{m} \setbar h^{T}\beta \geq 0, \forall \ \beta \in K\}$. \thref{Schaefer13 extension} constructs a stopping criterion $\bar{k}$ for the nested procedure such that $\aF = \aF^{\bar{k}}$. The stopping criterion is computable given a feasibility assumption associated with the continuous relaxation's dual.

\begin{proposition} \thlabel{Schaefer13 extension}
There exists a finite $\bar{k} \in \mathbb{Z}_+$ such that for all $k\geq \bar{k}$ and all $\beta \in \bH$, $\aF^k(b)=\aF(b)$. 
If $\{u \in K^{*} \setbar A^{T}u \geq 1\} \neq \emptyset$, then $\bar{k}$ can be computed by solving a (continuous) conic feasibility problem using the data $(A, 1, 0, K^{*}),$ and taking the maximum of $|\bH|$ inner-products.
\end{proposition}

\begin{proof} Since $\bH$ is a bounded set of integral points, $\bH$ is finite. For each $\beta \in \bH$, by \thref{Schaefer8 extension}{} there exists a $k_{\beta} \in \mathbb{Z}_+$ such that for all $k \geq k_{\beta}$, $F^k(\beta)=\aF(\beta)$. Set $\bar{k}= \max\{k_{\beta} \setbar \beta \in \bH\}$, then for all $k \geq \bar{k}$ and $\beta \in \bH, \aF^{k}(\beta) = \aF(\beta).$\\
Suppose $\{u \in K^{*} \setbar A^{T}u \geq 1\} \neq \emptyset$.
Consider the primal conic program $\max\{1^{T}x \setbar Ax \preceq_{K} 0, x \in \mathbb{R}^{n}_{+}\}$ and its dual $\min\{0^{T}u \setbar A^{T}u \geq 1, u \in K^{*}\}$. Note that the primal is feasible (the zero vector is a solution) and the dual is feasible (by assumption); by weak duality, both problems are bounded. Thus, both problems are feasible and bounded.\\
Let $\bar{u} \in \{u \in K^{*} \setbar A^{T}u \geq 1\}$ and $\bar{k} = \lceil \max\{\bar{u}^{T}\beta \setbar \beta \in \bH\}\rceil$, the latter of which is finite because $\bH$ is bounded. Obtaining $\bar{k}$ requires $|\bH|$ inner products of $m$-vectors. For each $\beta \in \bH$ and $k \in \mathbb{Z}_{+}$ such that $k \geq \bar{k}$, if $\bar{x} \in \aP(\beta)$, then $1^{T}\bar{x} \leq \bar{u}^{T}A\bar{x}$ because $A^{T}\bar{u} \geq 1 \geq 0$ and $\bar{x} \in \mathbb{R}^{n}_{+}$. Additionally, because $\beta - A\bar{x} \in K$ and $\bar{u} \in K^{*}$, $\bar{u}^{T}A\bar{x} \leq \bar{u}^{T}\beta$. By the definition of $\bar{k}, \bar{u}^{T}\beta \leq \bar{k} \leq k$. Hence, $\aP(\beta) = \emptyset$ if and only if $\aP^{k}(\beta) = \emptyset.$  
We conclude that $\aF^{k}(\beta) = \aF(\beta)$ for all $\beta \in \bH$ and all $k \geq \bar{k}$.
\end{proof}

\cite{Epelman2000} show that such a $\bar{u}$ in the proof of \thref{Schaefer13 extension} can be obtained in time polynomial in the problem data, the problem data's condition number, and a parameter that depends only on the dual cone $K^{*}$.

We conclude by noting that here are other sequences of functions that converge to $\aF$, and they may require fewer iterations to achieve convergence. Define the sequence of functions $G^{k}: \mathbb{R}^{m} \to \mathbb{R}$ by

\begin{equation}
G^k(\beta) \coloneqq 
\begin{cases*}
0 & if $ \{x \setbar Ax \preceq_K \beta, x\leq 2^k, x\in \mathbb{Z}^n_{+}\}\neq \emptyset, $ \\
-1 & if $\{x \setbar Ax \preceq_K \beta, x\leq 2^k, x\in \mathbb{Z}^n_{+}\} = \emptyset$.
\end{cases*}
\end{equation}

One can show $G^{k} \in \Gamma^{m}$ using a proof similar to that of \thref{schaefer5 extension}. Also, as $k$ increases, $G^{k}$ converges pointwise to $\aF$. \thref{Schaefer14 extension} indicates how a nested procedure can be constructed for $G^{k}$.

\begin{proposition} \thlabel{Schaefer14 extension}
For all $k\geq 1$,
$G^{k}(\beta) = \max\limits_{y \in \{0,1\}^{n}} G^{k-1}(\beta - 2^{k-1}Ay)$. \label{G_claim1}
\end{proposition}

\begin{proof} 
Observe that
$\{x\in \mathbb{Z}^n_+ \setbar Ax \preceq_K \beta, x_j\leq 2^k, \forAll j \in \{1,...,n\}\}=\cup_{y\in\{0,1\}^n}\{x\in \mathbb{Z}^n_+| Ax\preceq_K \beta, 2^{k-1}y_{j} \leq x_j \leq 2^{k-1}(y_{j}+1)\}$.

Consider $y \in \{0,1\}^{n}$. Suppose $\bar{x} \in \{x\in \mathbb{Z}^n_+ \setbar Ax \preceq_{K} \beta, 2^{k-1}y_{j} \leq x_j \leq 2^{k-1}(y_{j}+1)\}$. Let $\bar{z} = \bar{x} - 2^{k-1}y,$ then $A\bar{z} = A\bar{x} - 2^{k-1}Ay$. This implies that $A\bar{z} + 2^{k-1}Ay = A\bar{x} \preceq_{K} \beta$, which is true if and only if $A\bar{z} \preceq_{K} \beta - 2^{k-1}Ay$. Further, because $2^{k-1}y_{j} \leq \bar{x}_{j} \leq 2^{k-1}(y_{j}+1)$ for all $j \in \{1,...,n\}$, $0 \leq \bar{z} \leq 2^{k-1}$. Thus, $\bar{z} \in \{z \in \mathbb{Z}^{n}_{+} \setbar Az \preceq_K \beta - 2^{k-1}Ay, z \leq 2^{k-1}\}$.

Also, given $\hat{z} \in \{z \in \mathbb{Z}^{n}_{+} \setbar Az \preceq \beta - 2^{k-1}Ay, z \leq 2^{k-1}\}$, let $\hat{x} = \hat{z} + 2^{k-1}y$. It can be similarly shown that $\hat{x} \in \{x\in \mathbb{Z}^n_+| Ax\preceq_K \beta, 2^{k-1}y_j\leq x_j \leq 2^{k-1}(y_j+1)\}$.

Thus, for a given $y \in \{0,1\}^{n}$, $\{x\in \mathbb{Z}^n_+| Ax\preceq_K \beta, 2^{k-1}y_j\leq x_j \leq 2^{k-1}(y_j+1)\} \neq \emptyset$ if and only if $\{z \in \mathbb{Z}^{n}_{+} \setbar Az \preceq \beta - 2^{k-1}Ay, z \leq 2^{k-1}\} \neq \emptyset$. Moreover, $\{x \in \mathbb{Z}^{n}_{+} \setbar Ax \preceq_{K} \beta, x \leq 2^{k}\} = \emptyset$ if and only if for each $y \in \{0,1\}^{n}, \{z \in \mathbb{Z}^{n}_{+} \setbar Az \preceq_{K} \beta - 2^{k-1}Ay, z \leq 2^{k-1}\} = \emptyset$.

Therefore, $G^{k}(\beta) = 0$ if and only if $\max\limits_{y \in \{0,1\}^{n}} G^{k-1}(\beta - 2^{k-1}Ay) = 0$, and because $G^{l}(\beta) \in \{0, -1\}$ for all $\beta \in \mathbb{R}^{m}, l \in \mathbb{Z}_{+}$, and the result follows.
\end{proof}

\thref{Schaefer14 extension} shows that a similar approach to that used for $\aF^{k}$ is possible. Compared to $\aF^{k}$, iterating with $G^{k}$ requires fewer iterations as the restrictions on the size of feasible solutions are relaxed at an exponential rate. However, one can observe that these steps require more computation as one must search through the vectors $y \in \{0,1\}^{n}$.

\section{Conclusion}
In this paper, we establish theorems of the alternative for conic integer programs using superadditive duality. In addition, we provide a nested procedure to determine which integral right-hand sides in a bounded set are feasible.  Future directions of this work include theorems of the alternative for conic mixed-integer programs.  

\section*{Acknowledgments} The authors would like to thank David Mildebrath, Saumya Sinha, and Silviya Valeva of Rice University for their helpful comments. This research was supported by National Science Foundation grant CMMI-1826323.

\bibliographystyle{plainnat}
\bibliography{farkasbib.bib}

\newpage
\appendix
\appendixpage


\section{Proofs}
\noindent \textbf{Proposition~\ref{ourWeakDuality}.}
\textit{Let $x$ be a feasible solution to CIP($\beta$), and let $F$ be a feasible solution to \eqref{MoranDual}. Then $F(\beta) \geq c^{T}x$.}
\begin{proof}

Observe the following:
\begin{subequations}
\begin{align}
F(\beta) &\geq F(Ax) \label{wdProof1}\\
&\geq F\left(\sum\limits_{j = 1}^{n} a^{j}x_{j}\right) \label{wdProof2}\\
&\geq \sum\limits_{j = 1}^{n} F(a^{j}x_{j}) \label{wdProof3}\\
&\geq \sum\limits_{j = 1}^{n} F(a^{j})x_{j} \label{wdProof4}\\
&\geq \sum\limits_{j = 1}^{n} c_{j}x_{j}. \label{wdProof5}
\end{align}
\end{subequations}
\eqref{wdProof1} holds because $x$ is primal feasible and $F$ is dual feasible, which imply $\beta - Ax \in K$ and $F$ is nondecreasing with respect to $K$. \eqref{wdProof2}-\eqref{wdProof4} hold because $F$ is superadditive. \eqref{wdProof5} holds because $F(a^{j}) \geq c^{j}$, for all $j = 1,...,n$. 
\end{proof}

\noindent \textbf{Proposition~\ref{MoranIPDual}.} \citep{Moran2018} Consider CIP($\beta$), where $K$ is a regular cone. Then \eqref{MoranDual} is a strong dual to CIP($\beta$). 
\begin{proof} 
Let $\tilde{K} = \{(y,z) \in \mathbb{R}^{m + n} \setbar y \in K, z \in \mathbb{R}^{n}_{+}\}$, $\tilde{A} = [A^{T} \ -I]^{T} \in \mathbb{R}^{(m+n)\times n},$ where $I$ is the identity matrix in $\mathbb{R}^{n \times n}$, and $\tilde{\beta} = [\beta^{T} \ 0^{T}]^{T} \in \mathbb{R}^{m + n}$. Then CIP($\beta$) is equivalent to
\begin{equation}\label{MoranProof1}
\begin{aligned}
\inf \ &-c^{T}x\\
\text{s.t. } \ &\tilde{A}x \succeq_{-\tilde{K}} \tilde{\beta}\\
&x \in \mathbb{Z}^{n}.
\end{aligned}
\end{equation}

From \cite{Moran2018}, a strong dual to \eqref{MoranProof1} is 
\begin{equation}\label{MoranProof2}
\begin{aligned}
\sup \ &\tilde{G}(\beta)\\
\text{s.t. } \ &\tilde{G}(\tilde{a}^{j}) = -c_{j}, \text{ for all } j \in \{1,...,n\},\\
&\tilde{G}(0) = 0,\\
&\tilde{G} \text{ is subadditive and nondecreasing with respect to } -\tilde{K}.
\end{aligned}
\end{equation}

Further, \eqref{MoranProof2} is equivalent to 
\begin{equation}\label{MoranProof3}
\begin{aligned}
\alpha_{1} = \inf \ &\tilde{F}(\beta)\\
\text{s.t. } \ &\tilde{F}(\tilde{a}^{j}) = c_{j}, \text{ for all } j \in \{1,...,n\},\\
&\tilde{F}(0) = 0,\\
&\tilde{F} \text{ is superadditive and nondecreasing with respect to } \tilde{K}.
\end{aligned}
\end{equation}
This equivalence is due to a substitution ($-\tilde{G}$ for $\tilde{F}$), switching subadditive with superadditive, and by the relationship between $\tilde{G}$ is nondecreasing with respect to $-\tilde{K}$ and to $\tilde{F}$ is nondecreasing with respect to $\tilde{K}$. From \cite{Moran2018}, the optimal objective value of CIP($K$) is equal to $\alpha^{1}$, assuming feasibility. 

Consider
\begin{equation}\label{MoranProof4}
\begin{aligned}
\alpha_{2} = \inf \ &F(\beta)\\
\text{s.t.} \ &F(a^{j}) \geq c_{j}, \text{ for all } j \in \{1,...,n\},\\
&F(0) = 0,\\
&F \in \Gamma^{m}(K).
\end{aligned}
\end{equation}

We first show that $\alpha^{1} \geq \alpha^{2}$. Let $\widehat{F}$ be a feasible solution for \eqref{MoranProof3}, and define $\bar{F}: \mathbb{R}^{m} \to \mathbb{R}$ by $\bar{F}(y) = \widehat{F}(y, 0)$. By the feasibility of $\widehat{F}, \bar{F}(a^{j}) = \widehat{F}(a^{j}, 0) \geq \widehat{F}(\tilde{a}^{j}) = c_{j},$ for all $j = 1,...,n$. 
Additionally, $\bar{F}(0) = \widehat{F}(0,0) = 0$.  
Consider $y^{1}, y^{2} \in \mathbb{R}^{m}.$ Then $\bar{F}(y^{1}) + \bar{F}(y^{2}) = \widehat{F}(y^{1},0) + \widehat{F}(y^{2},0) \leq \widehat{F}(y^{1} + y^{2}, 0) = \bar{F}(y^{1} + y^{2}).$ Next, consider $y^{1}, y^{2}$ such that $y^{1} \succeq_{K} y^{2}$. Then $(y^{1},0) \succeq_{K} (y^{2}, 0)$, which implies $\bar{F}(y^{1}) = \widehat{F}(y^{1}, 0) \geq \widehat{F}(y^{2}, 0) = \bar{F}(y^{2})$. Hence, $\bar{F} \in \Gamma^{m}(K)$ and is feasible for \eqref{MoranProof4}.
Furthermore, $\bar{F}(b) = \widehat{F}(\tilde{b})$, which implies that $\alpha_{2} \leq \alpha_{1}$.

We now show that $\alpha^{1} \leq \alpha^{2}$. Consider an optimal solution $x^{*}$ of CIP($\beta$). Then $c^{T}x^{*} = \alpha^{1}$. For any feasible solution $\bar{F}$ of \eqref{MoranProof4},
\begin{align*}
\bar{F}(\beta) &\geq \bar{F}(Ax^{*})\\
&\geq \sum\limits_{j = 1}^{n} \bar{F}(a^{j}x^{*}_{j})\\
&\geq \sum\limits_{j = 1}^{n}\sum\limits_{k = 1}^{x^{*}_{j}} \bar{F}(a^{j})\\
&\geq \sum\limits_{j = 1}^{n}\sum\limits_{k = 1}^{x^{*}_{j}} c_{j}\\
&= c^{T}x^{*}\\
&= \alpha^{1}.
\end{align*}
\end{proof}

\section{Pseudocode}
We now explain the pseudocode briefly. Algorithm \ref{F_alg} is a more detailed version of Algorithm \ref{F_alg_short}, both of which describe the nested procedure in Section \ref{constructF}. Within Algorithm \ref{F_alg}, EVAL, EVAL-SPEC, LSM-POOL, and UPDATE-SETS are functions.

As defined in Section \ref{constructF}, the functions $\aF^{k}: \bH \to \mathbb{R}$ represent cardinality-constrained feasibility functions that approach $\aF$ as $k$ increases. The sets $\bB^{k}$ contain the level-set-minimal vectors of $\aF^{k}$ (vectors $\beta$ with $\aF^{k}(\beta) = 0$ that are minimal with respect to $K$), and they can be computed in a nested manner (one computes the sets $\mathbf{C}^{k}$ during this process). The set $\mathbf{H}$ starts with the user-specified finite set of right-hand sides, and at each iteration, it contains all considered right-hand sides that do not yield feasible problems at the current iteration. The set $\mathbf{S}$ contains right-hand sides for which one can guarantee feasibility at the current and all future iterations. During iteration $k$, Algorithm~\ref{F_alg} computes $\mathbf{C}^{k-1}$ using LSM-POOL. Because $\aF^{k}(\beta) \leq \aF^{l}(\beta)$ for any $k \leq l$, $\aF^{k}(\beta) = 0$ for all $\beta \in \mathbf{S}$; thus, $\aF^{k}(\beta)$ is inferred for all such $\beta$. EVAL and EVAL-SPEC are used to evaluate $\aF^{k}(\beta) = \max\{\aF^{k-1}(\beta), \max\limits_{j \in \{1,...,n\}} \aF^{k-1}(\beta - a^{j})\}$. UPDATE-SETS returns the updated solved and unsolved right-hand sides as well as the level-set-minimal vectors (when specified).

\begin{algorithm}
\caption{Construction of feasibility function $\aF^{\bar{k}}$}\label{F_alg}
\begin{algorithmic}[1]
\Procedure{MAIN}{}
\State Given: $A, \bar{k}, \bH, K$
\State $\bB^{0} \gets \{0\}$, $\mathbf{S} \gets \bH \cap K$, $\aF^{0} \gets 0$
\For{$\beta \in \mathbf{H}$ $\cap \ K$}
    \State $\aF^{0}(\beta) \gets 0$
\EndFor
\For{$\beta \in \mathbf{H} \backslash K$}
    \State $\aF^{0}(\beta) \gets -1$
\EndFor
\For{$k = 1,...,\bar{k}$}
    \State $\bB^{k} \gets \emptyset$
    \State $\mathbf{C}^{k-1} \gets $ LSM-POOL($\aF^{k-1}, \mathbf{B}^{k-1}, A, \bH$)
    \For{$\beta \in \mathbf{S}$}
        \State $\aF^{k}(\beta) \gets 0$
    \EndFor
    \For{$\beta \in \mathbf{C}^{k-1} \backslash (\mathbf{S} \cap (\bB^{k-1})^{c})$} \label{CkMinus1Line}
        \If{$\beta \not\in \mathbf{S}$}
            \State $\aF^{k}(\beta) \gets \text{EVAL-SPEC}(\aF^{k-1}, \beta, \bB^{k-1}, A)$
        \EndIf
        \If{$\aF^{k}(\beta) == 0$}
            \State $(\mathbf{S}, \mathbf{H}, \bB^{k}) \gets \text{UPDATE-SETS}(\mathbf{S}, \mathbf{H}, \beta, \bB^{k})$
        \EndIf
    \EndFor
    \For{$\beta \in \mathbf{H} \backslash (\mathbf{S} \cup \mathbf{C}^{k-1})$}
        \State $\aF^{k}(\beta) \gets \text{EVAL}(\aF^{k-1}, \beta, \bB^{k})$
        
        \If {$\aF^{k}(\beta) == 0$}
            \State $(\mathbf{S}, \mathbf{H}, \text{NULL}) \gets \text{UPDATE-SETS}(\mathbf{S}, \mathbf{H}, \beta, \text{NULL})$
            \EndIf
    \EndFor
\EndFor
\Return $\aF^{\bar{k}}$
\EndProcedure
\end{algorithmic}
\end{algorithm}

\begin{algorithm}
\caption{Evaluate $\aF^{k}(\beta)$ for $\beta \in \mathbf{C}^{k-1}$}
\begin{algorithmic}[1]
\Procedure{EVAL-SPEC}{}($\widetilde{F}, \tilde{\beta}, \widetilde{\bB}, \widetilde{A})$
    \State $z \gets \widetilde{F}(\tilde{\beta})$
    \If{$z == 0$}
        \State \Return $z$
    \EndIf
    \For{$j \in \{1,...,n\}, \bar{\beta} \in \widetilde{\bB}$}
        \If{$\bar{\beta} \preceq_{K} \tilde{\beta} - a^{j}$}
            \State $z \gets 0$
            \State \Return $z$
        \EndIf
    \EndFor
    \State \Return $-1$
\EndProcedure
\end{algorithmic}
\end{algorithm}

\begin{algorithm}
\caption{Evaluate $\aF^{k}(\beta)$ for $\beta \not\in \mathbf{C}^{k-1}$}
\begin{algorithmic}[1]
\Procedure{EVAL}{}($\widetilde{F}, \tilde{\beta}, \widetilde{\bB})$
    \State $z \gets \widetilde{F}(\tilde{\beta})$
    \If{$z == 0$}
        \State \Return $z$
    \EndIf
    \For{$\bar{\beta} \in \widetilde{\bB}$}
        \If{$\bar{\beta} \preceq_{K} \tilde{\beta}$}
            \State $z \gets 0$
            \State \Return $z$
        \EndIf
    \EndFor
    \State \Return $-1$
\EndProcedure
\end{algorithmic}
\end{algorithm}

\begin{algorithm}
\caption{Construct the set $\mathbf{C}^{k-1}$}
\begin{algorithmic}[1]
\Procedure{LSM-POOL}{}($\widetilde{F}, \widetilde{\mathbf{B}}, \widetilde{A}, \widetilde{\bH}$)
\State $\mathbf{C} \gets \emptyset$
\State $\widehat{\bH} \gets \{\beta \in \widetilde{\bH} \setbar \widetilde{F}(\beta) = -1 \text{ or } \beta \in \widetilde{\bB}\}$
\For{$\beta \in \widehat{\bH}$}
    \For{$j \in \{1,...,n\}$}
        \State $z_{j} \gets [\widetilde{F}(\beta - \widetilde{a}^{j}) == 0] \bigwedge [\neg(\beta - \widetilde{a}^{j} \in \widetilde{\bB})]$
        \If{$\bigwedge\limits_{j = 1}^{n} (\neg z_{j}) == \text{TRUE}$}
            \State $\mathbf{C} \gets \mathbf{C} \cup \{\beta\}$
        \EndIf
    \EndFor
\EndFor
\State \Return $\mathbf{C}$
\EndProcedure
\end{algorithmic}
\end{algorithm}

\begin{algorithm}
\caption{Update the level-set-minimal vector, solved vector, and unsolved vector sets}
\begin{algorithmic}[1]
\Procedure{UPDATE-SETS}{}($\widetilde{\mathbf{S}}, \widetilde{\mathbf{H}}, \tilde{\beta}, \widetilde{\bB}$)
\State $\widetilde{\mathbf{S}} \gets \widetilde{\mathbf{S}} \cup \{\tilde{\beta}\}$
\State $\widetilde{\mathbf{H}} \gets \widetilde{\mathbf{H}} \backslash \{\tilde{\beta}\}$
\If{$\widetilde{\bB} \neq \text{NULL}$}
    \State $\widetilde{\bB} \gets \widetilde{\bB} \cup \{\tilde{\beta}\}$
\EndIf
\Return{($\widetilde{\mathbf{S}}, \widetilde{\mathbf{H}}, \widetilde{\bB})$}
\EndProcedure
\end{algorithmic}
\end{algorithm}

\end{document}